\documentclass[letterpaper,final,10pt]{amsart}
\usepackage[utf8x]{inputenc}

\usepackage{amsmath}%
\usepackage{amsfonts}%
\usepackage{amssymb}%
\usepackage{graphicx}
\usepackage{functan}
\usepackage{mathrsfs}
\usepackage[obeyspaces]{url}
\usepackage{braket}
\usepackage[all,cmtip,curve,knot]{xy}
\usepackage{tikz}
\usetikzlibrary{arrows}
\tikzstyle{block}=[draw opacity=0.7,line width=1.4cm]
\usepackage{algorithm}
\usepackage{algorithmicx}
\usepackage{algpseudocode}
\usepackage{mathbbol}
\newtheorem{theorem}{Theorem}

\newtheorem{corollary}{Corollary}

\newtheorem{definition}{Definition}
\newtheorem{example}{Example}

\newtheorem{lemma}{Lemma}

\newtheorem{remark}{Remark}

\numberwithin{equation}{section}
\numberwithin{theorem}{section}
\numberwithin{lemma}{section}
\numberwithin{corollary}{section}
\numberwithin{definition}{section}
\numberwithin{example}{section}
\numberwithin{remark}{section}
\numberwithin{property}{section}
\numberwithin{proposition}{section}

\newcommand{\NS}[3][1]{#2_#1,\ldots,#2_#3}

\newcommand{\lift}[3][\pi]{\xymatrix{#2 \ar@{~>}[r]_{#1} & #3}}

\newcommand{\Ad}[2]{\mathrm{Ad}[#1](#2)}
\newcommand{\diag}[2][]{\mathrm{diag}_{#1}\left[ #2 \right]}
\newcommand{\N}[1]{\mathcal{N}(#1)}

\newcommand{\I}{\mathbb{1}}

\newcommand{\RR}{\mathbb{R}}
\newcommand{\CC}{\mathbb{C}}
\newcommand{\ZZ}{\mathbb{Z}}

\newcommand{\TT}[1][1]{{\mathbb{T}^{#1}}}

\newcommand{\disk}[1][2]{{\mathbb{D}^{#1}}}

\begin{document}
\title{Dynamical Deformation of Toroidal Matrix Varieties}
\author{Fredy Vides}
\address{School of Mathematics and Computer Science, Department of Applied Mathematics, 
Universidad Nacional Aut\'onoma de Honduras, Ciudad Universitaria, Tegucigalpa, Honduras.}

\email{fredy.vides@unah.edu.hn}



\date{\today}

\begin{abstract}
In this document we study the local connectivity of the sets whose elements are $m$-tuples of pairwise commuting normal matrix contractions. Given $\varepsilon>0$, we prove that there is $\delta>0$ such that for any two 
$m$-tuples of pairwise commuting normal matrix contractions $\mathbf{X}:=(X_1,\ldots,X_m)$ and $\tilde{\mathbf{X}}:=(\tilde{X}_1,\ldots,\tilde{X}_m)$ that are $\delta$-close with respect to some suitable distance $\eth$ in $(\mathbb{C}^{n\times n})^m$, we can find  a $m$-tuple of matrix paths (homotopies) connecting $\mathbf{X}$ to $\mathbf{\tilde{X}}$ relative to the intersection of some $\varepsilon,\eth$-neighborhood of $\mathbf{X}$ with the set of $m$-tuples of pairwise commuting normal matrix contractions. One of the key features of these matrix homotopies is that $\delta$ can be chosen independent of $n$.

Some connections with topology and numerical matrix analysis will be outlined as well.
\end{abstract}

\maketitle

\section{Introduction}
\label{intro}

In this document we study local deformation properties of what we call toroidal matrix varieties. Given a $m$-tuple of $n\times n$ matrices $\mathbf{X}\in (\CC^{n\times n})^m$ 
and a  suitable distance $\eth$ in $(\mathbb{C}^{n\times n})^m$ induced by the spectral norm (that will be defined below), let us write $N_\eth(\mathbf{X},r)$ to denote the set $\{\mathbf{Y}\in (\CC^{n\times n})^m | \eth(\mathbf{X},\mathbf{Y})\leq r\}$. We call $N_\eth(\mathbf{X},r)$ a $r,\eth$-neighborhood of $\mathbf{X}$. 

Given two matrices $X,Y\in (\CC^{n\times n})$, let us write $X\sim_h Y$ to indicate that there is a {\bf homotopy} between $X$ and $Y$, in other words, there is function $\gamma\in C([0,1],\CC^{n\times n})$ such that $\gamma(0)=X$ and $\gamma(1)=Y$ (where $C(X,Y)$ denotes the algebra of continuous functions on $X$ that take values in $Y$). Given $\varepsilon,\delta>0$, any two $m$-tuples of pairwise commuting normal matrix contractions $\mathbf{X}:=(X_1,\ldots,X_m)$ and $\tilde{\mathbf{X}}:=(\tilde{X}_1,\ldots,\tilde{X}_m)$ such that 
$\eth(\mathbf{X},\tilde{\mathbf{X}})\leq \varepsilon$, let us write $\mathbf{X}\sim_h \tilde{\mathbf{X}}$ to denote the homotopy 
induced by the homotopies of their components. We study the existence of homotopies $\mathbf{X}\sim_h \tilde{\mathbf{X}}$ relative to the intersection of some $\varepsilon,\eth$-neighborhood of $\mathbf{X}$ with the set of $m$-tuples of pairwise commuting normal matrix contractions.

We can think of each homotopy $X_j\sim_h \tilde{X}_j$ relative 
to $N_\eth(X_j,\varepsilon)$, as a noncommutative analogy of the family of curves ({\bf links}) connecting the point sets determined 
by the spectra $\sigma(X_j)$ and 
$\sigma(\tilde{X}_j)$ relative to the embedding of $\disk\times \TT[1]$ in $[-1,1]^2\times \TT$ (where the embedding is induced by the mapping $\disk\to B(0,1)\subseteq [-1,1]^2$). In a similar way we can interpret the induced homotopies 
$\mathbf{X}\sim_h \tilde{\mathbf{X}}$ as noncommutative analogies of the links connecting the joint spectra $\sigma(\mathbf{X})$ and 
$\sigma(\mathbf{\tilde{X}})$ relative to $[-1,1]^{2n}\times \TT$.

In this study we will use the $C^\ast$-algebraic technology developed in 
\cite{Rordam_Lin_Thm,Lin_Theorem,Vides_homotopies,Vides_matrix_words} to derive a connectivity technique that can be used to study and solve problems that appear in clustered matrix approximation (in the sense of \cite{Clustered_matrix_approximation}) and matrix dilation/compression problems (in the sense of \cite{Greenbaum_matrix_dilations, Holbrook_matrix_compressions} and \cite[\S9 and \S10]{Audenaert_Kittaneh_Open_Problems}).

In \cite{Vides_homotopies} we found some connections between the map $\Psi:=\mathrm{Ad}[W]$ described by L.\ref{Joint_spectral_variation_inequality_2} in \S\ref{notation} and the local homeomorphims that play a key role in the Kirby Torus Trick (introduced by R. Kirby in \cite{Torus_Trick_Kirby}) by moving sets of points in homeomorphic copies $\mathcal{T}^m\simeq \TT[m]$ of $\TT[m]$ {\bf \em just a little bit}, this analogy together with L.\ref{existence_of_almost_unit} provides a connection with {\bf \em topologically controlled linear algebra} (in the sense of \cite{CLA_Freedman}), which can be roughly described as the study of the relations between matrix sets and 
smooth manifolds that is performed by implementing techniques from geometric topology in the study of matrix approximation problems. The search for the previously mentioned analogies and connections was motivated by a question raised by M. H. Freedman regarding to the role played by the Kirby Torus Trick in linear algebra and matrix approximation theory.

In this document we build on the techniques developed in \cite{Vides_homotopies} and \cite{Vides_matrix_words} to study the analytic local connectivity properties of matrix representations of certain universal $C^\ast$-algebras. The deformation technology constructed for this purpose can be adapted using the techniques presented in \S\ref{main_results} to study the local behavior of approximate/numerical  solutions to 
matrix equations in the sense of \cite{applied_matrix_Homotopy_1,Chu_num_lin,Matrix_Poly_Dennis,Applied_Spectrum_Clustering_2,Applied_Spectrum_Clustering_3,
Applied_Dilated_Spectrum_Clustering_2}, that can be interpreted as local matrix homotopies that play the role of continuous analogies (in the sense of 
\cite{Chu_num_lin}) of the discrete/numerical procedures corresponding to spectral refinement algorithms (in the sense of \cite{Spectral_refinement_1,Spectral_refinement_2}) that can be used to refine the outputs of approximate simultaneous block diagonalization/transformation algorithms (in the sense of \cite{sim_block_diag,ASD_matrices}). 

The applications of the previously described refinement methods include singular saddle point problems that arise in computational fluid dynamics,
mixed finite element approximation of elliptic partial differential equations, optimization, optimal control, weighted least-squares problems, electronic networks, computer graphics and others
in scientific computing and engineering (in the sense of \cite{Applied_Spectrum_Clustering_2}), together with mathematical modelling of photonic crystals (in the sense of \cite{Applied_Spectrum_Clustering_3}) in physics. The results presented in \S\ref{main_results} can also be used to compute and to study the approximate and exact simultaneous block diagonalization of matrix $m$-tuples in the sense of \cite{sim_block_diag,ASD_matrices}, problems of this type appear in biomathematics, image processing and applied spectral graph theory.

Building on some of the ideas developed by M. A. Rieffel in \cite{Finite_groups_Rieffel,Vector_bundles_Rieffel}, by H. Lin in 
\cite{Lin_Theorem} and by P. Friis and M. R{\"o}rdam in \cite{Rordam_Lin_Thm}, we proved L.\ref{existence_of_almost_unit}, L.\ref{Existence_of_1D_PSRA} and L.\ref{Existence_of_mD_PSRA}, these results together with L.\ref{Joint_spectral_variation_inequality_2} provide us with the matrix approximation technology that we use to prove the main results.  The main results T.\ref{main_result}, C.\ref{main_corollary} and T.\ref{main_unitary_result} are presented in \S\ref{main_results} and some future directions are outlined in \S\ref{hints}.

\section{Preliminaries and Notation}
\label{notation}

Let $(X,d)$ be a metric space. Then we say that $\tilde{X}_\delta\subset X$ is a $\delta$-dense if for all $x\in X$ there exists $\tilde{x}\in \tilde{X}$ such that $d(x,\tilde{x})\leq\delta$. 
Given two compact subsets $X,Y$ of the complex plane, we will write $d_H(X,Y)$ to denote the Hausdorff distance between $X$ and $Y$ that is defined by the expression 
$d_H(X,Y):=\max\{sup_{x\in X} \inf_{y\in Y}|x-y|,sup_{y\in Y} \inf_{x\in X}|y-x|\}$. Let us denote by $\eth$ the function 
$\eth:(\CC^{n\times n})^m\times (\CC^{n\times n})^m\to \RR^+_0,(\mathbf{S},\mathbf{T})\mapsto \max_j \|S_j-T_j\|$. From here on $\|\cdot\|$ denotes the operator/spectral norm. Given two topological objects, we use the expression $X \simeq Y$ to denote a homeomorphism between them (a continuous function between the topological objects that has a continuous inverse function). We will write $\disk[m]$ and $\TT[m]$ to denote the $m$-dimensional closed unit disk and the $m$-dimensional torus respectively. We will write $B(x_0,r)$ to denote the closed ball 
$\{x\in \CC\: | \: |x-x_0|\leq r\}$.

We will write $M_n$ to denote the set $\mathbb{C}^{n\times n}$ of $n\times n$ complex matrices, the symbols $\mathbf{1}_n$ and $\mathbf{0}_n$ will be used to denote the identity matrix and the zero matrix in $M_n$ respectively. Given a matrix $A\in M_n$, we write $A^\ast$ to denote the conjugate transpose $\bar{A}^\top$ of $A$. A matrix $X\in M_n$ is said to be normal if $XX^\ast=X^\ast X$, a matrix $H\in M_n$ is said to be hermitian if $H^\ast=H$ and a 
matrix $U\in M_n$ such that $U^\ast U=UU^\ast=\mathbf{1}_n$ is called unitary. 

A hermitian matrix $P\in M_n$ such that $P=P^2$ is called a projector. Given two projectors $P$ and $Q$, if $PQ=QP=\mathbf{0}_n$ we say that $P$ and $Q$ are orthogonal. By an orthogonal partition of unity in $M_n$, we mean a finite set of orthogonal projectors $\{P_j\}$ in $M_n$ such that 
$\sum_j P_j=\mathbf{1}_n$. We will omit the explicit reference to $M_n$ when it is clear from the context. In this document by a matrix contraction we mean a matrix $X$ in $M_n$ for some $n\in \ZZ^+$ such that $\|X\|\leq 1$.

Given any two matrices $X,Y\in M_n$ we will write $[X,Y]$ and $\mathrm{Ad}[X](Y)$ to denote the operations 
$[X,Y]:=XY-YX$ and $\mathrm{Ad}[X](Y):=XYX^*$.

Given a matrix $A\in M_n$, we write $\sigma(X)$ to denote the set $\{\lambda\in \CC \: | \: \det(A-\lambda\mathbf{1}_n)=0\}$ of eigenvalues of $A$, the set $\sigma(A)$ is called the spectrum of $A$. Given a compact set $\mathbb{X}\subset \mathbb{C}$ and a subset $S\subseteq M_n$, let us dote by $S(\mathbb{X})$ the set 
$S(\mathbb{X}):=\{X\in S|\sigma(S)\subseteq \mathbb{X}\}$, for instance, the expression $\N{n}(\disk)$ is used to denote the set 
normal contractions in $M_n$.

\begin{definition}
Given $\varepsilon\geq 0$ and a matrix $X\in M_n$, we write $\sigma_\varepsilon(X)$ to denote the $\varepsilon$-Pseudospectrum of $X$ which is the set defined by the following relations.
\begin{eqnarray*}
\sigma_\varepsilon(X)&:=&\{\tilde{\lambda}\in \CC \: | \: \tilde{\lambda}\in \sigma(X+E), \mathrm{for \: some} \: E\in M_n \: \mathrm{with } \: \|E\|\leq \varepsilon\}\\
&=&\{\tilde{\lambda}\in \CC \: | \: \|X\mathbf{v}-\tilde{\lambda}\mathbf{v}\|\leq \varepsilon, \mathrm{for \: some} \: \mathbf{v}\in \CC^n \: \mathrm{with } \: \|\mathbf{v}\|=1\}
\end{eqnarray*}
\end{definition}

\begin{definition}[Semialgebraic Matrix Varieties]
\label{matrix_variety}
Given $J\in \ZZ^+$, a system 
of $J$ polynomials $\NS{p}{J}\in \Pi_{\braket{N}}=\mathbb{C}\Braket{\NS{x}{N}}$ in $N$ NC-variables $\NS{x}{N}\in \Pi_{\braket{N}}$ 
and a real number $\varepsilon\geq 0$, a particular matrix representation of the 
noncommutative semialgebraic set $\mathcal{Z}_{\varepsilon,n}(\NS{p}{J})$ 
described by 
\begin{equation}
 \mathcal{Z}_{\varepsilon,n}(\NS{p}{J}):=\Set{\NS{X}{N}\in M_{n} | \|p_j(\NS{X}{N})\|\leq \varepsilon, 1\leq j\leq J},
\end{equation}
will be called a {\bf $\varepsilon,n$-semialgebraic matrix variety} ($\varepsilon,n$-SMV), if $\varepsilon=0$ we can refer to the set as 
a {\bf matrix variety}. 
\end{definition}

\begin{example} Given any integer $n\geq 1$, let us set $\mathbf{N}:=\diag{n,n-1,\ldots,1}$, we will have that the 
set $Z_\mathbf{N}:=\{X\in M_n|[\mathbf{N},X]=0\}$ is a matrix variety. If for some $\delta>0$, we set now 
$Z_{\mathbf{N},\delta}:=\{X\in M_n|\|[\mathbf{N},X]\|\leq \delta\}$, the set $Z_{\mathbf{N},\delta}$ is 
a matrix semialgebraic variety.
\end{example}

\begin{example} The subset of $M_n^3$ described by,
\begin{equation}
\mathcal{Z}^3:=\left\{
(X_1,X_2,X_3)\in M_n^3
\left|
 \begin{array}{l}
  A_1X_1-X_2B_1=C_1,\\
  A_2X_3-X_2B_2=C_2
 \end{array}
\right. 
 \right\},
\label{matrix_variety_equation_1}
\end{equation}
where $A_1,A_2,B_1,B_2,C_1$ and $C_3$ are some fixed but arbitrary matrix contractions in $M_n$, is an algebraic matrix variety.
\end{example}

\begin{example} Given $\varepsilon>0$, the subset of $M_n^3$ described by,
\begin{equation}
\mathcal{Z}^3:=\left\{
(X_1,X_2,X_3)\in M_n^3
\left|
 \begin{array}{l}
  \|A_1X_1-X_2B_1-C_1\|\leq \varepsilon,\\
  \|A_2X_3-X_2B_2-C_2\|\leq \varepsilon,\\
  X_j-X_j^*=0, 1\leq j\leq m
 \end{array}
\right. 
 \right\},
\label{matrix_variety_equation_2}
\end{equation}
where $A_1,A_2,B_1,B_2,C_1$ and $C_3$ are some fixed but arbitrary matrix contractions in $M_n$, is a semialgebraic matrix variety.
\end{example}

\begin{remark}
Matrix sets of the form \ref{matrix_variety_equation_1} and \ref{matrix_variety_equation_2} provide a connection between matrix equations on words (in the sense of \cite{Johnson_matrix_words_spectrum,Vides_matrix_words}) and algebraic/semialgebraic matrix varieties. From this perspective, the computation/refinement  of matrix words corresponding to the numerical solution of matrix equations on words can be interpreted as a discrete analogy of local matrix homotopies like the ones constructed in the proofs of the results in \S\ref{main_results}.

\end{remark}

\begin{definition}[Curved and Flat matrix paths]
 Given any three hermitian matrices $-\mathbf{1}_n\leq H_1,H_2,H_3 \leq\mathbf{1}_n$ and a function $f\in C^1([-1,1])$, and given any four normal contractions 
 $D_1,\ldots,D_4$ in $M_n$, with $D_2=\Ad{e^{\pi i H_1}}{D_1}$, $D_3=f(H_2)$ and $D_4=f(H_3)$. Let us consider the paths $\breve{Z}(t):=\Ad{e^{\pi i t H_1}}{D_1}$ 
 and $\bar{V}(t):=f(tH_3+(1-t)H_2)$. We will say that $\breve{Z}$ is a {\em \bf curved} interpolating path for $D_1,D_2$ and we will say that the path 
 $\bar{V}$ is a {\em \bf flat} interpolating path for $D_3,D_4$.
\end{definition}

\begin{definition}[$\circledast$ operation]
 Given two matrix paths $X,Y\in C([0,1],M_n)$ we write ${X\circledast Y}$ to denote the concatenation of $X$ and $Y$, which is 
 the matrix path defined in terms of $X$ and $Y$ by the expression,
 \[
  {X\circledast Y}_s:=
  \left\{
  \begin{array}{l}
   X_{2s},\:\: 0\leq s\leq \frac{1}{2},\\
   Y_{2s-1},\:\: \frac{1}{2}\leq s\leq 1.
  \end{array}
  \right.  
 \]
\end{definition}

\begin{definition}[Local matrix deformations $x \rightsquigarrow_{\varepsilon,S} y$]
 Given two matrices $x,y\in M_n$ we write $x\rightsquigarrow y$ if there is a matrix path $z\in C([0,1],M_n)$ such that 
 $z_0=x$ and $z_1=y$, if there is a $\varepsilon$-local matrix homotopy 
 $X\in C([0,1],M_n)$ between $x$ and $y$ relative to the set $S$, we will write $x\rightsquigarrow_{\varepsilon,S} y$ and will omit the explicit reference to $S$ when it is clear from the context.
\end{definition}

It is often convenient to have $N$-tuples (or $2N$-tuples) of matrices with real spectra. For this purpose we use the following construction, 
initiated by McIntosh and Pryde. If $X=(\NS{X}{N})$ is a $N$-tuple of $n$ by $n$ matrices then we can always decompose $X_j$ in the form 
$X_j=X_{1j}+iX_{2j}$ where the $X_{kj}$ all have real spectra. We write 
$\pi(X):=(X_{11},\ldots,X_{1N},X_{21},\ldots,X_{2N})$ and call $\pi(X)$ a partition of $X$. If the $X_{kj}$ all commute we say that 
$\pi(X)$ is a commuting partition, and if the $X_{kj}$ are simultaneously triangularizable $\pi(X)$ is a triangularizable partition. If 
the $X_{kj}$ are all semisimple (diagonalizable) then $\pi(X)$ is called a semisimple partition.

We say that $N$ normal matrices $\NS{X}{N}\in M_n$ are {\em simultaneously diagonalizable} if there is a unitary matrix 
$Q\in M_n$ such that 
$Q^* X_jQ$ is diagonal for each $j=1,\ldots,N$. In this case, for $1\leq k\leq n$, let 
$\Lambda^{(k)}(X_j):=(Q^*X_jQ)_{kk}$ the $(k,k)$ element of $Q^*X_jQ$, and set 
$\Lambda^{(k)}(\NS{X}{N}):=(\Lambda^{(k)}(X_1),\ldots,\Lambda^{(k)}(X_N))$ in $\CC^N$. The set
\[
 \Lambda(\NS{X}{N}):=\{\Lambda^{(k)}(\NS{X}{N})\}_{1\leq k\leq N}
\]
is called the joint spectrum of $\NS{X}{N}$. We will write $\Lambda(X_j)$ to denote the diagonal matrix representation of the $j$-component of $\Lambda(\NS{X}{N})$, in other words 
we will have that
\[
 \Lambda(X_j)=\diag{\Lambda^{(1)}(X_j),\ldots,\Lambda^{(n)}(X_j)}.
\]

The following result (\cite[L.4.1]{Vides_homotopies}) was proved in \cite{Vides_homotopies}.

\begin{lemma}\label{Joint_spectral_variation_inequality_2}
 Given $\varepsilon>0$ there is $\delta=\frac{1}{K_m}\varepsilon> 0$ such that, for any two $N$-tuples of pairwise commuting normal 
 matrices $\mathbf{x}=(\NS{x}{N})$ and $\mathbf{y}:=(\NS{y}{N})$ 
 such that $\eth(\mathbf{x},\mathbf{y})\leq \delta$, there is a unitary matrix $W$ such that $[\Ad{W}{x_j},y_j]=0$ and 
 $\max\{\|\Ad{W}{x_j}-y_j\|,\|\Ad{W}{x_j}-x_j\|\}\leq \varepsilon$, for each $1\leq j\leq N$.
\end{lemma}

\begin{remark}
The constant $K_m$ in the statement of L.\ref{Joint_spectral_variation_inequality_2} depends only $m$.
\end{remark}

\section{Local Deformation of Normal Contractions}
\label{main_results}

Applying functional calculus on commutative $C^\ast$-algebras we have that the joint spectrum $\sigma(X,Y)$ of any two commuting hermitian matrix contractions $X,Y\in M_n$ is contained in $[-1,1]^2$, moreover if $\|X+iY\|\leq 1$ we will have that 
$\sigma(X,Y)\subseteq B(0,1)\simeq \disk$.

\begin{lemma}
\label{existence_of_almost_unit}
 Given $r\in \ZZ^+$ and $\nu>0$, there is $\delta:=\hat{\delta}(\nu,r)>0$ such that for any unitary $W$ and any normal contraction $D$ in $M_n$ for $n\geq 2$, if $D=\sum_{j=1}^r \alpha_jP_j$ is diagonal for $2\leq r\in \ZZ$ and $\NS{\alpha}{r}\in \disk$, the 
 set $\{P_j\}$ is an orthogonal partition of unity, and $\alpha_j\neq \alpha_k$ whenever $k\neq j$, then there 
 is a unitary matrix $Z\in M_n$ such that $[Z,D]=0$ and $\|\I_n-WZ\|\leq\nu$ whenever $\|WDW^*-D\|\leq \delta$.
\end{lemma}
\begin{proof}
Let $\mu:=\|WDW^*-D\|=\|WD-DW\|$. Then there are $r$ continuous functions $\ell_1,\ldots,\ell_r\in C(\disk)$ such that 
$P_k:=\ell_k(D)$. By continuity one can find $\delta(\nu,r)>0$ such that, $\|WP_kW^\ast-P_k\|=\|W\ell_k(D)W^\ast-\ell_k(D)\|= \|\ell_k(WDW^\ast)-\ell_k(D)\|\leq \nu/(\sqrt{2}r)<1/(\sqrt{2}r)$, whenever $\mu\leq \delta$. Since $\nu<1$, by perturbation theory of $C^*$-algebras we will have that there is a unitary $W_j$ such that $\|\mathbf{1}-W_j\|\leq \sqrt{2}\|WP_kW^\ast-P_k\|\leq \nu/r$ and $W_j^\ast P_j W_j=WP_jW^\ast$, and this implies that $W_jWP_j=P_jW_j W$ for each $j$. Since $\{P_j\}$ is an orthogonal partition of unity, we will have that $\tilde{W}:=\sum_j W_jWP_j$ is a 
unitary matrix which satisfies the commutation relation $[\tilde{W},D]=0$ together with the normed inequalities.
\begin{eqnarray}
\|W-\tilde{W}\|&=&\|W\sum_jP_j -\sum_j W_jWP_j\|\\
               &=&\|\sum_j((\mathbf{1}_n-W_j)WP_j)\|\\
               &\leq&\sum_j \|\mathbf{1}_n-W_j\|\\
               &\leq&\sum_j \frac{1}{r}\nu=r(\frac{1}{r}\nu)=\nu
\end{eqnarray}
Let us set $Z:=\tilde{W}^\ast$ then $\|\mathbf{1}_n-WZ\|=\|\tilde{W}-W\|\leq \nu$. This completes the proof.
\end{proof}

\begin{definition}[Pseudospectral retractive approximant] Given $\delta>0$ and any matrix $X\in M_n$, we say that the matrix $\tilde{X}\in M_n$ is a $\delta$-Pseudospectral retractive approximant ({\bf $\delta$-PSRA}) of $X$ if $\|X-\tilde{X}\|\leq \delta$ and 
$\sigma(\tilde{X})$ is $\delta$-dense in $\sigma(X)$.
\end{definition}

\begin{lemma}
\label{Existence_of_1D_PSRA}
Given $\delta>0$ and any hermitian matrix $X\in M_n$ such that $\|X\|\leq 1$, there is a hermitian $\delta$-PSRA $\tilde{X}_\delta$ of $X$ 
such that $[X,\tilde{X}]=\mathbf{0}_n$. Moreover, there are $N_\delta$ distinct points $\{x_1,\ldots,x_{N_\delta}\}$ and an orthogonal partition of unity $\{P_1,\ldots,P_{N_\delta}\}$ such that $\tilde{X}_\delta=\sum_j x_j P_j$.
\end{lemma}
\begin{proof}
Let us suppose that $n\geq |\sigma(X)|\geq 2$, as the proof is clear for scalar multiples of $\mathbf{1}_n$. Since $\sigma(X)\subseteq [-1,1]$, we can assume for simplicity that $M_\delta:=1+(\delta)^{-1} \in \mathbb{Z}^+$. Then the finite set 
$\hat{R}_\delta(X):=\{x_k:=-1+2(k-1)\delta\: | \: 1\leq k\leq M_\delta\}$ is $\delta$-dense in $[-1,1]$ with $|R_\delta(X)|=M_\delta$. Let us set $S_\delta(X):=\{\check{x}_k:=-1+2(k-2)\delta\: | \: 1\leq k \leq N_{\delta}+1\}$ and 
$P_k:=\hat{\chi}_{(\check{x}_k,\check{x}_{k+1}]}(X)$ for each $1\leq k\leq M_\delta$, where $\hat{\chi}_{(c,d]}$ denotes a continuous representation/approximation of the characteristic function $\chi_{(c,d]}$ of $(c,d]$ on some set $S_{c,d}$ such that 
$(c,d]\subseteq S_{c,d}\subset \mathbb{R}$. Then there is $R_\delta(X):=\{x_j\}\subseteq \hat{R}_\delta(X)$ that is $\delta$-dense in $\sigma(X)$ with $x_j\in (\check{x}_{k(j)},\check{x}_{k(j)+1}]$ for each $x_j\in R_\delta(X)$ and some $\check{x}_{k(j)},\check{x}_{k(j)+1}\in S_\delta(X)$, and there is a set $P_\delta(X):=\{P_1,\ldots,P_{N_\delta}\}\subseteq M_n\backslash \{\mathbf{0}_n\}$ that is an orthogonal partition of unity such that $[P_j,X]=0$ and for each $1\leq k\leq N_\delta$ there is  $x_{j(k)}\in R_\delta(X)$ such that $\|XP_k-x_{j(k)}P_k\|\leq \delta$ with $x_{k(j)}\neq x_{k(l)}$ whenever $j\neq l$. By the previous facts and spectral variation of normal matrices (in the sense of 
\cite{Bhatia_mat_book,Pryde_Inequalities}) we will have that if we set $\tilde{X}_\delta:=\sum_k x_{j(k)}P_k$ with $x_{j(k)}\in R_\delta(X)$ for each $k$, then $[X,\tilde{X}_\delta]=\mathbf{0}_n$ and 
$d_H(X,\tilde{X}_\delta)\leq \|X-\tilde{X}_\delta\|\leq\max_k \|XP_k-x_{j(k)}P_k\|\leq \delta$. If necessary we can renumber the elements of $R_\delta(X)$ according to the elements of $P_\delta(X)$. This completes the proof.
\end{proof}

\begin{remark}
We will refer to the finite sets $R_\delta(X)$ and $S_\delta(X)$ as {\em \bf representation} and {\em \bf support} grids respectively. The finite set $P_\delta(X)$ will be called a {\bf $\delta$-projective} decomposition for $X$.
\end{remark}

\begin{lemma}
\label{Existence_of_mD_PSRA}
Given $\delta>0$, $m\geq 1$ and any $m$-tuple of pairwise commuting hermitian matrix contractions $\mathbf{X}=(X_1,\ldots,X_m)\in M_n^m$, there is a $m$-tuple of pairwise commuting hermitian matrix contractions $\tilde{\mathbf{X}}_\delta=(\tilde{X}_{\delta,1},\ldots,\tilde{X}_{\delta,m})\in M_n^m$ such that $[X_j,\tilde{X}_{\delta,k}]=\mathbf{0}_n$ and $\tilde{X}_{\delta,j}$ is a $\delta$-PSRA of $X_j$ for each $1\leq j,k\leq m$. Moreover, there are $N_\delta$ distinct points $\{x_{k,1},\ldots,x_{k,N_\delta}\}$ and an orthogonal partition of unity $\{P_1,\ldots,P_{N_\delta}\}$ such that $\tilde{X}_{\delta,k}=\sum_j x_{k,j} P_j$ for each $k$.
\end{lemma}
\begin{proof}
Since the joint spectrum $\sigma(\mathbf{X})$ of $\mathbf{X}$ is a subset of $[-1,1]^m$ we can apply a similar procedure to the one implemented in the proof of L.\ref{Existence_of_1D_PSRA} to find for each $X_j$ a representation grid $\hat{R}_{\delta}(X_j)$ together with a support grid $S_{\delta}(X_j)$ and an associated 
$\delta$-projective decomposition $P_\delta(X_j)$ such that $|P_\delta(X_j)|=N_{\delta,j}:=|\hat{R}_{\delta}(X_j)|$, we will have that 
the set $P_\delta(\mathbf{X}):=\{P_1,\ldots,P_{N_\delta}\}=\{P_{1,j_1}P_{2,j_2}\cdots P_{m,j_m}\:|\: P_{k,j_k}\in P_\delta(X_k), \: 1\leq j_k\leq N_{\delta,k}, \: 
1\leq k \leq m\}$ is an orthogonal partition of unity, by setting $(1+\delta^{-1})^m\geq N_{\delta}:=|P_\delta(\mathbf{X})|$ and $\nu:=1/(N_\delta-1)$, will have that $N_\delta\geq N_{\delta,k}$ for each $k$, that $\nu\leq \delta$ and that there is a set 
$R_{\delta}(X_j)\subseteq \{x_{j,k}=2(k-1)\nu-1\:|\: 1\leq k\leq N_\delta\}$ of distinct points that is $\delta$-dense in $\sigma(X_j)$ for each $j$ and such that for each $X_k$ and each $P_j\in P_\delta(\hat{\mathbf{X}})$, there is $x_{k,j}\in R_\delta(X_j)$ such that $\|X_kP_j-x_{k,j}P_j\|\leq \delta$. Moreover,  
$[X_j,P_k]=\mathbf{0}_n$ for each $1\leq j\leq m$ and each $1\leq k \leq N_\delta$. By a similar argument to the one implemented in the proof of L.\ref{Existence_of_1D_PSRA} it can be seen that the matrix $\tilde{X}_{\delta,k}:=\sum_{j}x_{k,j}P_j$ is a 
commuting hermitian $\delta$-PSRA of $X_j$ for each $j$. This completes the proof.
\end{proof}

\begin{definition}[Uniform piecewise analytic local connectivity in $M_n^m$]
We will say that a matrix variety $\mathbb{Z}^m \subseteq M_n^m$ is uniformly locally piecewise analytically connected ({\bf ULPAC}) if given $\varepsilon>0$, there is $\delta>0$ such that for any two $m$-tuples $\mathbf{X}$ and $\mathbf{Y}$ in $\mathbb{Z}^m$ such that 
$\eth(\mathbf{X},\mathbf{Y})\leq \delta$ ($\mathbf{X}$ and $\mathbf{Y}$ are $(\delta_\varepsilon,\eth)$-close), we have that $\mathbf{X}\sim_h \mathbf{Y}$ relative to $N_\eth(X,\varepsilon)\cap \mathbb{Z}^m$.
\end{definition}

\begin{remark}
It is important to remark that $\varepsilon$ and $\delta$ must not depend on $n$.
\end{remark}

Let us define the toroidal matrix varieties whose local connectivity will be studied in this document. 

\begin{definition}[$\mathbb{I}^m(n)$]
The matrix variety $\mathbb{I}^m(n)$ defined by
\[
\mathbb{I}^m(n):=
\left\{
(X_1,\ldots,X_m)\in M_n^m \left|
\begin{array}{l}
[X_j,X_k]=\mathbf{0}_n\\
X_j-X_j^\ast=\mathbf{0}_n\\
\|X_j\|\leq 1
\end{array}
, 1\leq j,k \leq m\right.
\right\}
\]
will be called the $m$ {\em \bf matrix cube}.
\end{definition}

\begin{definition}[$\mathbb{D}^m(n)$]
The matrix variety $\mathbb{D}^m(n)$ defined by
\[
\mathbb{D}^m(n):=
\left\{
(Z_1,\ldots,Z_m)\in M_n^m \left|
\begin{array}{l}
[Z_j,Z_k]=[Z_j,Z_j^\ast]=\mathbf{0}_n\\
\|Z_j\|\leq 1
\end{array}
, 1\leq j,k \leq m\right.
\right\}
\]
will be called the $m$ {\em \bf matrix disk}.
\end{definition}

\begin{definition}[$\mathbb{T}^m(n)$]
The matrix variety $\mathbb{T}^m(n)$ defined by
\[
\mathbb{T}^m(n):=
\left\{
(U_1,\ldots,U_m)\in M_n^m \left|
\begin{array}{l}
[U_j,U_k]=\mathbf{0}_n\\
U_jU_j^\ast=U_j^\ast U_j=\mathbf{1}_n
\end{array}
, 1\leq j,k \leq m\right.
\right\}
\]
will be called the $m$ {\em \bf matrix torus}.
\end{definition}

\begin{remark}
It is important to notice that the components of two $m$-tuples $\mathbf{X}=(X_1,\ldots,X_m)$ and $\mathbf{Y}=(Y_1,\ldots,Y_m)$ in $\mathbb{I}^m(n)$, $\mathbb{D}^m(n)$ or $\mathbb{T}^m(n)$ need not to satisfy the commutation relations $[X_j,Y_k]=\mathbf{0}_n$ for each $j,k$ in general.
\end{remark}

\begin{theorem}
\label{main_result}
The matrix variety $\mathbb{I}^m(n)$ is uniformly locally piecewise analytically connected.
\end{theorem}
\begin{proof}
Without loss of generality we can assume that the components of $\mathbf{Y}$ are diagonal matrices. As a consequence of L.\ref{Joint_spectral_variation_inequality_2} we will have that given $\nu_\delta>0$, there is $\delta:=\frac{1}{K_m}\nu_\delta>0$ such that for any $m$-tuple $\mathbf{X}\in \mathbb{I}^m(n)$ that satisfies $\eth(\mathbf{X},\mathbf{Y})\leq \delta$, there exists a 
unitary matrix $W$ such that $[\Ad{W}{X_j},Y_j]=0$ and 
 $\max\{\|\Ad{W}{X_j}-Y_j\|,\|\Ad{W}{X_j}-X_j\|\}\leq \nu_\delta$, for each $1\leq j\leq N$. Since $\sigma(X_j)\subseteq [-1,1]\supseteq \sigma(Y_j)$ for each $j$. By joint spectral variation (in the sense of \cite{Pryde_Inequalities}) and by applying L.\ref{Joint_spectral_variation_inequality_2}  again, we will have that $W^\ast\Lambda(X_j)W=X_j$ and $\|W^\ast\Lambda(X_j)W-\Lambda(X_j)\|\leq K_m\delta$ for each $j$. Then by L.\ref{Existence_of_mD_PSRA} there is a $m$-tuple of pairwise commuting 
 hermitian $\delta$-PSRA $\tilde{\mathbf{X}}_{\delta}$ of $\mathbf{\Lambda(X)}:=(\Lambda(X_1),\ldots,\Lambda(X_m))$, together with $m$ representation grids $R_\delta(X_j):=\{x_{k,1},\ldots,x_{k,N_\delta}\}$ that are $\delta$-dense in $\sigma(X_j)=\sigma(\Lambda(X_j))$ for each $j$, and a projective decomposition $P_\delta(\mathbf{\Lambda(X)}):=\{P_1,\ldots,P_{N_\delta}\}$ such that $\tilde{X}_{\delta,j}=\sum_j x_{k,j}P_j$, $[X_{\delta,j},\Lambda(X_j)]=\mathbf{0}_n$ and $\|X_{j}-W^\ast\tilde{X}_{\delta,j}  W\|\leq \delta$. Let us 
 set $\hat{\mathbf{X}}_\delta:=(W^\ast X_{\delta,1}W,\ldots,W^\ast X_{\delta,m}W)$. We will have that.
\begin{eqnarray}
\|\hat{X}_{\delta,j}-\tilde{X}_{\delta,j}\|&\leq& \eth(\hat{\mathbf{X}}_\delta,\tilde{\mathbf{X}}_\delta)\\
&\leq& \eth(\hat{\mathbf{X}}_\delta,\mathbf{X})+\eth(\mathbf{X},\mathbf{\Lambda(X)})+\eth(\mathbf{\Lambda(X)},\tilde{\mathbf{X}}_\delta)\\
&\leq&(K_m+2)\delta
\label{proof_inequality_1}
\end{eqnarray} 
Let us set $N_{\varepsilon}:=|P_\delta(\mathbf{Y})|$. If $N_\varepsilon=1$, there are four flat hermitian paths $\bar{X}^{1,j}(t):=X_j+t(\hat{X}_{\delta,j}-X_j)$, $\bar{X}^{2,j}(t):=\hat{X}_{\delta,j}+t(\tilde{X}_{\delta,j}-\hat{X}_{\delta,j})$, $\bar{X}^{3,j}(t):=\tilde{X}_{\delta,j}+t(\Lambda({X}_{j})-\tilde{X}_{\delta,j})$ and $\bar{X}^{4,j}(t):=\Lambda(X_j)+t(Y_{j}-\Lambda(X_{j})$ such that the matrix path $\hat{X}^j:=[(\bar{X}^{1,j}\circledast \bar{X}^{2,j}) \circledast \bar{X}^{3,j}]\circledast \bar{X}^{4,j}\in C([0,1],I^m(n))$ solves the problem $X_j \rightsquigarrow_{\varepsilon_1} Y_j$ for each $j$ with 
$\varepsilon_1=(2K_m+3)\delta$ and this implies that $\mathbf{X}\sim_h \mathbf{Y}$ relative to $N_\eth(\mathbf{X},\varepsilon_1)\cap I^m(n)$. If $N_\varepsilon\geq 2$, then we can apply L\ref{existence_of_almost_unit} to find a unitary $\hat{W}:=ZW$ such that 
$\hat{W}^\ast \tilde{X}_{\delta,j} \hat{W}=W^\ast \tilde{X}_{\delta,j} W=\hat{X}_{\delta,j}$ and that $\|\mathbf{1}_n-\hat{W}\|\leq \varepsilon_2(N_\varepsilon,m,\delta)<2$, this implies that one can find a hermitian matrix contraction $\hat{K}$ such that $W=e^{i \pi\hat{K}}$. We have that the curved path $\breve{X}^{2,j}(t):=\Ad{e^{-i \pi(1-t)\hat{K}}}{\tilde{X}_{\delta,j}}$ solves the problem $\hat{X}_{\delta,j} \rightsquigarrow_{\varepsilon_2} \tilde{X}_{\delta,j}$ for each $j$ with $\varepsilon_2:=2\varepsilon_2(N_\varepsilon,m,\delta)$. Let us set $\hat{X}^{j}:=[(\bar{X}^{1,j}\circledast \breve{X}^{2,j}) \circledast \bar{X}^{3,j}]\circledast \bar{X}^{4,j}\in C([0,1],\mathbb{I}^m(n))$ with $\bar{X}^{1,j}$, $\bar{X}^{2,j}$ and $\bar{X}^{3,j}$ defined as before. Then $\hat{X}^{j}$ solves the problem $X_j \rightsquigarrow_{\varepsilon_3} Y_j$ for each 
$j$ with $\varepsilon_3:=\varepsilon_2+(K_m+1)\delta$, and this implies that $\mathbf{X}\sim_h\mathbf{Y}$ relative to $N_\eth(\mathbf{X},\varepsilon_3)\cap \mathbb{I}^m(n)$. Let us set $\varepsilon:=\max\{\varepsilon_1,\varepsilon_3\}$. This completes the proof.
\end{proof}

\begin{corollary}
\label{main_corollary}
The matrix variety $\mathbb{D}^m(n)$ is uniformly locally piecewise analytically connected.
\end{corollary}
\begin{proof}
Given any two $m$-tuples $\mathbf{Z}$ and $\mathbf{S}$ in $\mathbb{D}^m(n)$, such that $\eth(\mathbf{Z},\mathbf{S})\leq r/2$ for some 
$r>0$, there are two semisimple commuting hermitian partitions $\pi(\mathbf{Z})$ and $\pi(\mathbf{S})$ of $\mathbf{Z}$ and $\mathbf{S}$ respectively, such that $\eth(\pi(\mathbf{Z}),\pi(\mathbf{S}))\leq r$. By the previously described fact the result follows by applying T.\ref{main_result} to the corresponding  semisimple commuting hermitian partitions of any two $(\delta_\varepsilon,\eth)$-close $m$-tuples in $\mathbb{D}^m(n)$.
\end{proof}

\begin{theorem}
\label{main_unitary_result}
The matrix variety $\mathbb{T}^m(n)$ is uniformly locally piecewise analytically connected.
\end{theorem}
\begin{proof}
We have that for any two unitary matrices $U,V\in M_n$ such that $[U,V]=\mathbf{0}_n$ there are hermitian matrices $H_u,H_v$ such that 
the flat path $U(t):=e^{i\pi(H_u+t(H_v-H_u))}$ satisfies the interpolating conditions $U(0)=U$, $U(1)=V$, together with the constraints $[U(t),U]=[U(t),V]=\mathbf{0}_n$ and $U(t) U(t)^\ast=U(t)^\ast U(t)=\mathbf{1}_n$ for each $t\in[0,1]$. We also have that for any unitary $U$ and any hermitian contraction $K$ the curved path $V(t):=\Ad{e^{i\pi tK}}{U}$ preserves commutativity and is a unitary for each $t\in [0,1]$. By the previously described facts the result follows by modifying the proof of T.\ref{main_result} conveniently for any two $(\delta_\varepsilon,\eth)$-close $m$-tuples in $\mathbb{T}^m(n)$.
\end{proof}

\section{Hints and Future Directions}
\label{hints}
The implications of the main results in \S\ref{main_results} in the numerical solution of matrix equations on words will be further explored, some applications 
to the analysis of approximation/refinement procedures that can be implemented to solve matrix equations on words and eigenvalue problems in the sense of 
\cite{Spectral_refinement_1,Spectral_refinement_2,Randomized_Homotopy,Applied_Dilated_Spectrum_Clustering_1,
Johnson_matrix_words,Johnson_matrix_words_spectrum} will be the subject of further study as well.

The approximation and connectivity technology developed in this paper has a natural connection to approximate simultaneous diagonalization of $m$-tuples of matrices (in the sense of \cite{ASD_matrices}) and to normal matrix approximation of almost normal matrices in the sense of \cite{Greenbaum_matrix_dilations,Aspects_of_nonnormality}. The development of numerical algorithms to perform these tasks will be the subject of future communications.

The application and extension of the results in \S\ref{main_results} to the study of normal and near normal compressions of normal matrices (in the sense of \cite{Greenbaum_matrix_dilations,Holbrook_matrix_compressions} and \cite[\S9 and \S10]{Audenaert_Kittaneh_Open_Problems}) will be the subject of further study as well.

\section*{Acknowledgement}
I am very grateful with the Erwin Schr{\"o}dinger Institute for Mathematical Physics of the University of Vienna, for the outstanding hospitality and support during my visit to participate in the research program on Topological phases of quantum matter in 
August of 2014.  Much of the research reported in this document was carried out while I was visiting the Institute.

I am grateful with Terry Loring, Alexandru Chirvasitu, Moody Chu, Marc Rieffel, Stan Steinberg, Jorge Destephen and Concepci\'on Ferrufino, for several interesting questions and 
comments that have been very helpful for the preparation of this document.


\section*{References}





\begin{thebibliography}{00}
 \bibitem{Audenaert_Kittaneh_Open_Problems} {\sc K. M. R. Audenaert and F. Kittaneh.} {\em Problems and Conjectures in Matrix and Operator Inequalities.} arXiv:1201.5232v3 [math.FA] 2012.
 
  \bibitem{Spectral_refinement_1} {\sc M. Ahues, A. Largillier, F. Dias d'Almeida and P. B. Vasconcelos.} {\em 
 Spectral refinement on quasi-diagonal matrices.} Linear Algebra Appl. 401 (2005) 109–117.

 \bibitem{Spectral_refinement_2} {\sc M. Ahues, F. Dias d'Almeida, A. Largillier and P. B. Vasconcelos.} {\em 
 Spectral refinement for clustered eigenvalues of quasi-diagonal matrices.} Linear Algebra Appl. 413 (2006) 394–402.

 \bibitem{Randomized_Homotopy} {\sc D. Armentano and F. Cucker.} {\em A Randomized Homotopy for the Hermitian 
 Eigenpair Problem.} Found Comput Math (2015) 15:281-312.
 

 
  \bibitem{Bhatia_mat_book} {\sc R. Bhatia.} {\em Matrix Analysis.} Gaduate Texts in Mathematics 169. Springer-Verlag. 1997.
  
   \bibitem{applied_matrix_Homotopy_1} {\sc L. Chen, L. Han, and L. Zhou.} {\em Computing Tensor Eigenvalues via Homotopy 
 Methods.} SIAM J. Matrix Anal. Appl. Vol. 37, No. 1, pp. 290–319, 2016
 
 \bibitem{Applied_Dilated_Spectrum_Clustering_1} {\sc C. Davis, W. M. Kahan and H. F. 
 Weinberger.} {\em Norm-Preserving Dilations and Their Applications to Optimal 
 Error Bounds.} SIAM J. Numer. Anal. Vol. 19, No. 3, June 1982.

  \bibitem{Chu_num_lin} {\sc M. T. Chu.} {\em Linear Algebra Algorithms as Dynamical Systems.} Acta Numer. (2008), pp. 001-086. 2008.
 
 \bibitem{Matrix_Poly_Dennis} {\sc J. E. Dennis, J. F. Traub and R. P. Weber.} {\em The Algebraic Theory 
 of Matrix Polynomials.} SIAM J. Numer. Anal. Vol. 13, No. 6, December 1976.
 

  


 \bibitem{Rordam_Lin_Thm} {\sc P. Friis and M. R{\"o}rdam.} Almost commuting self-adjoint matrices - a short proof of Huaxin Lin's theorem. J. Reine Angew. 
 Math., 479:121–131, 1996.


 
  \bibitem{CLA_Freedman} {\sc M. H. Freedman and W. H. Press.} {\em Truncation of Wavelet Matrices: Edge Effects and the Reduction of 
 Topological Control} Linear Algebra Appl. 2:34:1-19 (1996)


 
 \bibitem{Greenbaum_matrix_dilations}{\sc A. Greenbaum, T. Caldwell and K. Li.} {\em Near Normal Dilations of Nonnormal Matrices and Linear Operators.} SIAM. J. Matrix Anal. Appl., Vol. 31, No. 4, pp. 1365-1381, 2016. 
 
 \bibitem{Johnson_matrix_words} {\sc C. J. Hillar and C. R. Johnson.} 
 {\em Symmetric Word Equations in two Positive Definite Letters.} Proceedings. Amer. Math. Soc., 
 Volume 132, Number 4, Pages 945-953. S 0002-9939(03)07163-6 2003.
 
 \bibitem{Holbrook_matrix_compressions} {\sc J. Holbrook, N. Mudalige and R. Pereira.} {\em Normal Matrix Compressions.} Oper. Matrices, Vol. 7, No. 4 (2013), 849-864. 
 
 \bibitem{Applied_Spectrum_Clustering_3} {\sc T.-M. Huang, W.-W. Lin and W. Wang.} {\em A hybrid Jacobi–Davidson method for interior cluster eigenvalues
with large null-space in three dimensional lossless Drude dispersive
metallic photonic crystals.} Comput. Phys. Comm. 207 (2016) 221-231.

\bibitem{Aspects_of_nonnormality} {\sc M. Huhtanen.} {\em Aspects of nonnormality for iterative methods.} Linear Algebra Appl. 394 (2005) 119-144.
 
 \bibitem{Johnson_matrix_words_spectrum} {\sc C. R. Johnson and C. Hillar.} {\em 
 Eigenvalues of  Words  in Two  Positive De nite Letters.} SIAM J. Matrix Anal. Appl., 23 (2002), 916-928. MR 2003e:81071
 


\bibitem{Applied_Dilated_Spectrum_Clustering_2} {\sc Y.-F. Ke and C.-F. Ma.} {\em 
Spectrum analysis of a more general augmentation block preconditioner for generalized 
saddle point matrices.} BIT Numer Math (2016) 56:489-500 DOI 10.1007/s10543-015-0570-0.

 \bibitem{Torus_Trick_Kirby} {\sc R. C. Kirby.} {\em Stable homeomorphisms and the annulus conjecture.} Ann. Math., Second Series, Vol. 89, 
 No. 3 (May, 1969), pp. 575-582
  
  \bibitem{Applied_Spectrum_Clustering_2} {\sc Z.-Z. Liang and G.-F. Zhang.} {\em Convergence behavior of generalized parameterized Uzawa method for singular saddle-point problems.} J. Comput. Appl. Math. 311 (2017) 293–305.
 

 
 \bibitem{Lin_Theorem} {\sc H. Lin.} {\em Almost Commuting Selfadjoint Matrices and Applications.} In Operator algebras and their applications (Waterloo, ON, 1994/1995), volume 13 of Fields Inst. Commun., pages 193-233. Amer. Math. Soc., Providence, RI, 1997.
 
 \bibitem{Vides_homotopies} {\sc T. A. Loring and F. Vides.} {\em Local Matrix Homotopies and Soft Tori.} arXiv:1605.06590 [math.OA]. 2016.
 
 \bibitem{sim_block_diag} {\sc T. Maehara and K. Murota.} {\em Algorithm for Error-Controlled Simultaneous Block-Diagonalization of Matrices.} SIAM J. Matrix Anal. Appl., 
 Vol. 32, No. 2, pp. 605-620, 2011.
 

 
  \bibitem{ASD_matrices} {\sc K. C. O'Meara and C. Vinsonhaler.} {\em On approximately simultaneously diagonalizable matrices.} Linear Algebra Appl. 412 (2006) 39-74.


 
 \bibitem{Pryde_Inequalities} {\sc A. J. Pryde.} {\em inequalities for the Joint Spectrum of Simultaneously Triangularizable Matrices.} 
 Proc. Centre Math. Appl., Mathematical Sciences Institute, The Australian National University (1992), 
 196-207.
 
 \bibitem{Finite_groups_Rieffel} {\sc M. A. Rieffel.} {\em Actions of Finite Groups on C*-Algebras.} Math. Scand. 47 (1980), 157-176. 1980.
 
 \bibitem{Vector_bundles_Rieffel} {\sc M. A. Rieffel.} {\em Vector Bundles and Gromov-Hausdorff Distance.} 
 J. K-theory 5(2010), 39-103.
 
 \bibitem{Clustered_matrix_approximation} {\sc B. Savas and I. S. Dhillon.} {Clustered Matrix Approximation.} SIAM J. Matrix Anal. Appl. Vol. 37, No. 4, pp. 1531-1555, 2016.
 

  \bibitem{Vides_matrix_words} {\sc F. Vides.} {\em Local Deformation of Matrix Words.} arXiv:1608.08562 [math.OA]

\end{thebibliography}
\end{document}